\documentclass{amsart}
\usepackage{CommonMath,enumerate,tikz, geometry, paralist, longtable, doi, url}
\usetikzlibrary{decorations.pathmorphing,shapes}
\usepackage[center]{subfigure}
\usepackage[square,numbers]{natbib}
\newcommand{\comment}[1]{}

\newcommand {\Z}{{\bf Z}}
\newcommand {\Q}{{\bf Q}}

\renewcommand {\P}{{\bf P}}
\newcommand {\from}{{\colon}}
\newcommand{\into}{{\hookrightarrow}}
\newcommand{\onto}{\twoheadrightarrow}
\newcommand{\orb}[1]{{\mathcal #1}}     
\renewcommand {\o}[1]{\overline{#1}}    
\newcommand{\dual}[1]{#1^{\vee}}
\newcommand{\isom}{\stackrel\sim\longrightarrow}

\DeclareMathOperator{\spec}{Spec}
\DeclareMathOperator{\proj}{Proj}
\DeclareMathOperator{\br}{br}
\DeclareMathOperator{\Sym}{Sym}
\DeclareMathOperator{\id}{id}
\DeclareMathOperator{\tr}{tr}

\DeclareMathOperator{\F}{\mathbf F}

\DeclareMathOperator{\Pic}{Pic}
\DeclareMathOperator{\ch}{ch}
\DeclareMathOperator{\td}{td}

\newcommand{\s}{\mathbf S}

\usepackage{tikz}
\tikzset{math/.style = {execute at begin node=$, execute at end node=$}}
\tikzset{map/.style = {font=\scriptsize}}

\renewcommand{\o}{\overline}
\DeclareMathOperator{\mult}{mult}

\title{Sharp slope bounds for sweeping families of trigonal curves}
\author{Anand Deopurkar \and Anand Patel}

\begin{document}
\maketitle
\begin{abstract}
  We establish sharp bounds for the slopes of curves in $\o M_g$ that
  sweep the locus of trigonal curves, proving Stankova-Frenkel's
  conjectured bound of $7+6/g$ for even $g$ and obtaining the bound
  $7+20/(3g+1)$ for odd $g$. For even $g$, we find an explicit
  expression of the so-called Maroni divisor in the Picard group of
  the space of admissible triple covers. For odd $g$, we describe the
  analogous extremal effective divisor and give a similar explicit
  expression.
\end{abstract}
\section{Introduction}
The seminal papers of
\citet{harris82:_kodair_dimen_of_modul_space_of_curves}, and
\citet{eisenbud87:_kodair} prove that the moduli space $\o M_g$ of
Deligne--Mumford stable curves of genus $g$ is of general type for $g
\geq 23$. This result marks the beginning of the quest for describing
the log-canonical models of $\o M_g$, namely the spaces
\[ \o M_g(\alpha) = \proj \bigoplus_{n \geq 0} H^0\left(\o {\orb M}_g,
  n(K_{\o {\orb M}_g} + \alpha \delta)\right).\] Understanding these
spaces and the maps between them is an active area of
research with important contributions by Hassett, Hyeon, Keel,
Morrison and several others
\citep{fedorchuk12:_alter_compac_modul_spaces_curves}.

As a step towards this goal, it is important to understand the stable
base loci of the linear systems $|K_{\o {\orb M}_g} + \alpha \delta|$,
or equivalently the linear systems $|s\lambda - \delta|$, where
$\lambda$ is the Hodge class, $\delta$ the total boundary class and
$s = 13/(2-\alpha)$. To test whether a given $X \subset \o M_g$ lies
in such a base locus, we may cover it by curves of high slope
($\delta/\lambda$); if $X$ can be covered by curves of slope $s_0$
then it must be in the base locus of $|s\lambda-\delta|$ for $s <
s_0$. We are thus led to the following question.
\begin{question}\label{q:slope}
  For a subvariety $X \subset \o M_g$, what is the highest $s$ such
  that $X$ is covered by curves of slope $s$?
\end{question}
The first result in this direction came from the work of \citet{cornalba88:_divis}. They answered it for the locus of hyperelliptic curves, getting the slope bound $8+4/g$. Next, the work of \citet{Stankova-Frenkel00:_Modul_Of_Trigon_Curves} suggested that the answer for the locus of trigonal curves is $7+6/g$, at least if $g$ is even. This is now known to be correct, thanks to the work of \citet[Theorem~1.6]{barja09:_slopes} and the more recent work of \citet[Theorem~1.2, Corollary~3.2]{fedorchuk12:_stabil_hilber_point_canon_curves}. Since the loci of Brill--Noether special curves are expected to make up the base loci of the log-canonical linear systems, \autoref{q:slope} is particularly interesting for such loci in $\o M_g$.

In this article, we settle the question for the locus of trigonal curves of any genus. We reprove the $7+6/g$ bound for even $g$ and obtain the bound $7+20/(3g+1)$ for odd $g$.
\begin{theorem}[\autoref{thm:sweeping_slope_sharp_even}, \autoref{thm:sweeping_slope_sharp_odd} in the main text]\label{thm:main}
  Let $g \geq 4$. Denote by $\o T_g$ the closure in $\o M_g$ of
  the locus of smooth trigonal curves. Set
  \[ s_g = \begin{cases}
    7 + 6/g & \text{ if $g$ is even},\\
    7 + 20/(3g+1) & \text{ if $g$ is odd}.
  \end{cases}
  \]
  Then,
  \begin{enumerate}
  \item $\o T_g$ is covered by curves of slope $s_g$. \label{curve}
  \item If $B \subset \o T_g$ is a curve passing through a general
    point of $\o T_g$, then its slope is at most $s_g$. In
    particular, $\o T_g$ cannot be covered by curves of higher slope. \label{divisor}
  \end{enumerate}
  The word ``general'' in the second statement means away from an
  (explicitly given) effective divisor.
\end{theorem}
Our approach is classical, but we use a modern ingredient to make it
feasible. For \eqref{curve}, we explicitly construct curves with slope $s_g$ that sweep $\o T_g$. For \eqref{divisor}, we explicitly
construct an effective divisor $D \subset \o T_g$ such that
\[ a[D] = s_g \lambda - \delta - \sum_{i} a_i\delta_i,\] with $a$
positive and $a_i$ non-negative. It follows that any curve $B \subset
\o T_g$ whose generic point avoids $D$ and the boundary divisors
$\delta_i$ must have $(s_g \lambda - \delta) \cdot B \geq 0$, yielding
the bound in \eqref{divisor}. For even $g$, the divisor $D$
is the closure of the locus of smooth trigonal curves that embed in an
unbalanced scroll. For odd $g$, it is the closure of the
locus of smooth trigonal curves that embed in $\F_1$ and are tangent to
the directrix. Thus, an equivalent formulation of our result is that
$(s_g\lambda - \delta)$ spans an edge of the effective cone in the
$\langle \lambda, \delta \rangle$ plane in the Picard group of $\o T_g$.

Here is the road-map for the proof. In \autoref{sec:compactification},
we describe a compactification of $T_g$ using admissible covers which
is more convenient than $\o T_g$. In \autoref{sec:formulas}, we recall
a classical description of flat triple covers in terms of linear
algebraic data on the base and express the standard divisor classes
$\lambda$, $\kappa$, $\delta$ in terms of this data. In
\autoref{sec:curve}, we finish off \autoref{thm:main}~\eqref{curve}
using the linear algebraic description. This section does not require
the more delicate orbifold considerations in the previous sections. In
\autoref{sec:divisor}, we prove \autoref{thm:main}~\eqref{divisor} by
computing the divisor class of $D$ in the Picard group of the
admissible cover compactification.

We would like to highlight the modern ingredient that makes
our approach successful. To compute the divisor class of $D$, we use
test-curves that intersect the boundary divisors at generic
points. This lets us circumvent having to identify the points of the
boundary that lie in $D$. Constructing enough such test-curves,
however, is challenging. The basis of our trigonal constructions---the
linear algebraic description---fails for admissible covers, which are
not necessarily flat. The idea that rescues us is the formulation of
admissible covers as flat covers of orbi-nodal curves due to
\citet*{acv:03}. Another piece of modern technology, namely the
Grothendieck--Riemann--Roch theorem for Deligne--Mumford stacks
\citep{edidin12:_rieman_roch_delig_mumfor} makes this computationally
feasible.

We work over an algebraically closed field $k$ of characteristic
zero. \emph{Genus} always means \emph{arithmetic genus}. By the genus
of an orbi-curve, we mean the genus of its coarse space. A
\emph{cover} is a representable, finite, flat morphism. A triple cover
$\phi \from C \to \P^1$ is \emph{balanced} (resp. \emph{unbalanced})
if the degrees of the two summands of the vector bundle $\phi_* O_C /
O_{\P^1}$ differ by at most $1$ (resp. at least
$2$). Projectivizations are spaces of one dimensional
\emph{quotients}, following Grothendieck's convention. Throughout,
$\zeta$ denotes a primitive $n$th root of unity.

The work in this paper began after conversations with Maksym Fedorchuk
when the authors were graduate students of Joe Harris. We thank both
of them for their kindness and appreciation. During the preparation of
the article, we learned about the results of \citet{beorchia12:_m}
about similar questions. We thank them for sharing their manuscripts
and thoughts.

\section{A compactification of $T_g$ and its boundary}\label{sec:compactification}
Let $\o{\orb M}_g$ be the moduli stack of Deligne--Mumford stable curves of genus $g$.  Let $\o {\orb H}^3_g$ be the moduli stack of twisted admissible covers of genus $g$ and degree three with unordered branch points where at most two branch points are allowed to coincide. In symbols
\[ \o{\orb H}^3_g = \{(\phi \from \orb C \to \orb P, \orb P \to P)\},\]
where
\begin{itemize}[\qquad]
\item $P$ is a connected nodal curve of genus zero,
\item $\orb P$ is a balanced orbi-nodal curve as in \citep{acv:03} with coarse space $\orb P \to P$,
\item $\orb C$ is a connected orbi-nodal curve of genus $g$,
\item $\phi \from \orb C \to \orb P$ is a twisted admissible cover of degree $3$ in the sense of \citep{acv:03}, 
\item $(P, \br\phi)$ is a $w$-stable pointed curve in the sense of \citep{hassett03:_modul} with $w = (\frac12, \dots, \frac12)$. 
\end{itemize}
For the convenience of the reader, we briefly recall the meanings of some of the terms. A \emph{balanced orbi-nodal curve} $\orb P \to P$ is a  modification of $P$ at the nodes; \'etale locally around the nodes, $\orb P \to P$ has the form
\[ [(\spec k[u,v]/uv)/\mu_n] \to \spec k[x,y]/xy,\]
where $\mu_n$ acts by $\zeta \from u \mapsto \zeta u$ and $\zeta \from v \mapsto \zeta^{-1}v$. A \emph{twisted admissible cover} $\phi \from \orb C \to \orb P$ is a representable, flat, finite morphism of degree $d$, \'etale over the nodes and the generic points of the components of $\orb P$ such that the associated map $\orb P \setminus \br\phi \to B\s_d$ is representable. The last condition is simply to ensure that the stack structure on $\orb P$ is the minimal required. Note that the induced map on the coarse spaces $\phi \from C \to P$ is an admissible cover in the sense of \citet{harris82:_kodair_dimen_of_modul_space_of_curves}. Finally, the  $w$-stability condition means that $\mult_p \br\phi \leq 2$ for all $p \in P$, and $\omega_P\left( \frac{\br\phi}{2}\right)$ is ample. 

By standard methods, one can prove that $\o {\orb H}^3_g$ is a connected, smooth, proper, Deligne--Mumford stack with a projective coarse space. We refer the reader to \citep{deopurkar12:_compac_hurwit} for the technical details. The compactification $\o {\orb H}^3_g$ and the more customary compactification by twisted admissible covers (where the branch points are distinct) are isomorphic on the level of coarse spaces. The functorial description that allows two branch points to collide, however, is quite convenient.

The open locus $\orb H^3_g$ where $P$ is smooth and $\br\phi$ reduced is the locus of simply branched trigonal curves. When two branch points coincide, $\orb C$ may develop either a node or a triple ramification point. In any case, the coarse space of $\orb C$ is at worst a nodal curve of genus $g$. By contracting unstable rational tails, we get a morphism $\o{\orb H}^3_g \to \o{\orb M}_g$. We denote by $\lambda$ the Hodge line bundle and by $\delta$ the total boundary divisor of $\o {\orb M}_g$. Abusing notation, we denote their pullbacks to $\o {\orb H}^3_g$ by the same letters.

The boundary $\o {\orb H}^3_g \setminus \orb H^3_g$ is a union of several divisors which we pictorially list in \autoref{fig:boundary}. For clarity, we draw the admissible cover $C \to P$ instead of the twisted admissible cover $\orb C \to \orb P$.
\begin{figure}[h]
  \centering
  \subfigure[$A$]{
    \begin{tikzpicture}[yscale=.25, xscale=1, thick]
      \draw [smooth, tension=1]
      plot coordinates{(1,1) (0.2, 1) (-0.2,-1) (-1,-1)};
      \draw (-1,-5) -- (1, -5);
      \path (0,-2.7) edge [->, thin] (0,-4.2);
    \end{tikzpicture}
  }\qquad \qquad
  \subfigure[$\Delta$]{
    \begin{tikzpicture}[yscale=.25, xscale=1, thick]
      \draw [smooth, tension=.5]
      plot coordinates{(1,1) (-1, -1) (-1,1) (1,-1) (1,-2) (-1,-2)} ;
      \draw (-1,-5) -- (1, -5);
      \path (0,-2.7) edge [->, thin] (0,-4.2);
    \end{tikzpicture}
  }\qquad\qquad
  \subfigure[$H$]{
    \begin{tikzpicture}[yscale=.25, xscale=1, thick]
      \draw [smooth, tension=2]
      plot coordinates{(1,1) (0, .5) (1,-.5)}
      plot coordinates{(-1,1) (0,.5) (-1,-.8)}
      plot coordinates{(-1,0) (1, -1.5)};
      \draw (-1,-5) -- (1, -5);
      \path (0,-2.7) edge [->, thin] (0,-4.2);
      \draw (1, .5) node [right] {$g$};
      \draw (1, -2) node [right] {$0$};
    \end{tikzpicture}
  }

  \subfigure[$\Delta_1(g_1,g_2)$ \newline with $g_1+g_2 = g-2$ and $g_i \geq 0.$]{
    \begin{tikzpicture}[yscale=.25, xscale=.8, thick]
      \path[dotted] (-1.5,-8) rectangle (3.5,4);
      \begin{pgflowlevelscope}{\pgftransformrotate{20}}
        \draw [smooth, tension=1.5]
        plot coordinates{(-1,1) (0, .2) (-1,-1) (1,-1)}
        plot coordinates{(1.1,1) (0, .2) (1,0)};
        \draw (-1,-5) -- (.7, -5);
        \path (-.5,-2.7) edge [->, thin] (-.5,-4.2);
      \end{pgflowlevelscope}
      \begin{pgflowlevelscope}{ 
          \pgftransformxscale{-1}
          \pgftransformxshift{-1.4cm}           
          \pgftransformrotate{20}}
        \draw [smooth, tension=1.5]
        plot coordinates{(-1,1) (0, .2) (-1,-1) (1,-1)}
        plot coordinates{(1.1,1) (0, .2) (1,0)};
        \draw (-1,-5) -- (.7, -5);
        \path (-.5,-2.7) edge [->, thin] (-.5,-4.2);
      \end{pgflowlevelscope}
      \draw (-1.3,0) node {$g_1$};
      \draw (3.3,0) node  {$g_2$};
    \end{tikzpicture}
  }
  \quad
  \subfigure[$\Delta_2(g_1,g_2)$ \newline with $g_1+g_2 = g-1$ and $g_i \geq 0$]{
    \begin{tikzpicture}[yscale=.25, xscale=.8, thick]
      \path[dotted] (-1.5,-8) rectangle (3.5,4);
      \begin{pgflowlevelscope}{\pgftransformrotate{20}}
        \draw [smooth, tension=1.5]
        plot coordinates{(1.2,1) (0, .2) (.82,-.7) (-1,-1)}
        plot coordinates{(-1.1,1) (0, .2) (-1,0)};
        \draw (-1,-5) -- (.7, -5);
        \path (-.5,-2.7) edge [->, thin] (-.5,-4.2);
      \end{pgflowlevelscope}
      \begin{pgflowlevelscope}{ 
          \pgftransformxscale{-1}
          \pgftransformxshift{-1.4cm}           
          \pgftransformrotate{20}}
        \draw [smooth, tension=1.5]
        plot coordinates{(1.2,1) (0, .2) (.82,-.7) (-1,-1)}
        plot coordinates{(-1.1,1) (0, .2) (-1,0)};
        \draw (-1,-5) -- (.7, -5);
        \path (-.5,-2.7) edge [->, thin] (-.5,-4.2);
      \end{pgflowlevelscope}
      \draw (-1.3,-1) node {$g_1$};
      \draw (3.3,0) node  {$g_2$};
    \end{tikzpicture}
  }
  \quad
  \subfigure[$\Delta_3(g_1,g_2)$ \newline with $g_1+g_2 = g$ and $g_i \geq 1$]{
    \begin{tikzpicture}[yscale=.25, xscale=.8, thick]
      \path[dotted] (-1.5,-8) rectangle (3.5,4);
      \begin{pgflowlevelscope}{\pgftransformrotate{20}}
        \draw [smooth, tension=3]
        plot coordinates{(-1,1) (1, 0) (-1,-1)}
        plot coordinates{(-1,0) (1,0)};
        \draw (-1,-5) -- (.7, -5);
        \path (-.5,-2.7) edge [->, thin] (-.5,-4.2);
      \end{pgflowlevelscope}
      \begin{pgflowlevelscope}{ 
          \pgftransformxscale{-1}
          \pgftransformxshift{-1.5cm}           
          \pgftransformrotate{20}}
        \draw [smooth, tension=3]
        plot coordinates{(-1,1) (1, 0) (-1,-1)}
        plot coordinates{(-1,0) (1,0)};
        \draw (-1,-5) -- (.7, -5);
        \path (-.5,-2.7) edge [->, thin] (-.5,-4.2);
      \end{pgflowlevelscope}
      \draw (-1.3,0) node {$g_1$};
      \draw (3.3,0) node  {$g_2$};
    \end{tikzpicture}
  }

  \subfigure[$\Delta_4(g_1,g_2)$ \newline with $g_1+g_2 = g-1$ and $g_2\geq 1$]{
    \begin{tikzpicture}[yscale=.25, xscale=.8, thick]
      \path[dotted] (-1.5,-8) rectangle (3.5,4);
      \begin{pgflowlevelscope}{\pgftransformrotate{20}}
        \draw [smooth, tension=1.5]
        plot coordinates{(-1,1) (0, .2) (-1,-1) (1,-1)}
        plot coordinates{(1.1,1) (0, .2) (1,0)};
        \draw (-1,-5) -- (.7, -5);
        \path (-.5,-2.7) edge [->, thin] (-.5,-4.2);
      \end{pgflowlevelscope}
      \begin{pgflowlevelscope}{ 
          \pgftransformxscale{-1}
          \pgftransformxshift{-1.4cm}           
          \pgftransformrotate{20}}
        \draw [smooth, tension=2]
        plot coordinates{(1.1,1) (0, .2) (1.1,0)}
        plot coordinates{(-1,.8) (0, .2) (-1,0)}
        plot coordinates{(1,-1) (-1, -1)};
        \draw (-1,-5) -- (.7, -5);
        \path (-.5,-2.7) edge [->, thin] (-.5,-4.2);
      \end{pgflowlevelscope}
      \draw (-1.3,0) node {$g_1$};
      \draw (3.2,0) node  {$g_2$};
    \end{tikzpicture}
  }
  \quad
  \subfigure[$\Delta_5(g_1,g_2)$ \newline with $g_1+g_2 = g$ and $g_2 \geq 1$]{
    \begin{tikzpicture}[yscale=.25, xscale=.8, thick]
      \path[dotted] (-1.5,-8) rectangle (3.5,4);
      \begin{pgflowlevelscope}{\pgftransformrotate{20}}
        \draw [smooth, tension=1.5]
        plot coordinates{(1.2,1) (0, .2) (.82,-.7) (-1,-1)}
        plot coordinates{(-1.1,1) (0, .2) (-1,0)};
        \draw (-1,-5) -- (.7, -5);
        \path (-.5,-2.7) edge [->, thin] (-.5,-4.2);
      \end{pgflowlevelscope}
      \begin{pgflowlevelscope}{ 
          \pgftransformxscale{-1}
          \pgftransformxshift{-1.4cm}           
          \pgftransformrotate{20}}
        \draw [smooth, tension=2]
        plot coordinates{(-1,0) (.85, -.6) (-1,-1)}
        plot coordinates{(-1,1) (1.2,1)};
        \draw (-1,-5) -- (.7, -5);
        \path (-.5,-2.7) edge [->, thin] (-.5,-4.2);
      \end{pgflowlevelscope}
      \draw (-1.2,-2) node {$g_1$};
      \draw (3,-2) node  {$g_2$};
    \end{tikzpicture}
  }
  \quad
  \subfigure[$\Delta_6(g_1,g_2)$ \newline with $g_1+g_2 = g$ and $g_i \geq 1$]{
    \begin{tikzpicture}[yscale=.25, xscale=.8, thick]
      \path[dotted] (-1.5,-8) rectangle (3.5,4);
      \begin{pgflowlevelscope}{\pgftransformrotate{20}}
        \draw [smooth, tension=2]
        plot coordinates{(1.1,-1) (0, -.3) (1.1,0)}
        plot coordinates{(-1,-.8) (0, -.3) (-1,0)}
        plot coordinates{(1.1,1) (-1, 1)};
        \draw (-1,-5) -- (.7, -5);
        \path (-.5,-2.7) edge [->, thin] (-.5,-4.2);
      \end{pgflowlevelscope}
      \begin{pgflowlevelscope}{ 
          \pgftransformxscale{-1}
          \pgftransformxshift{-1.4cm}           
          \pgftransformrotate{20}}
        \draw [smooth, tension=2]
        plot coordinates{(1.1,1) (0, .2) (1.1,0)}
        plot coordinates{(-1,.8) (0, .2) (-1,0)}
        plot coordinates{(1,-1) (-1, -1)};
        \draw (-1,-5) -- (.7, -5);
        \path (-.5,-2.7) edge [->, thin] (-.5,-4.2);
      \end{pgflowlevelscope}
      \draw (-1.3,-2) node {$g_1$};
      \draw (3.2,0) node  {$g_2$};
    \end{tikzpicture}
  }
\caption{Boundary divisors of $\o{\orb H}_g^3$}
\label{fig:boundary}
\end{figure}
The pictures are, of course, symbolic and only show the generic points. They are to be read as follows. A generic point of $A$ represents a trigonal curve with a triple ramification point. A generic point of $\Delta$ represents an irreducible trigonal curve with a node. A generic point of $H$ represents a hyperelliptic curve of genus $g$ with a rational tail attached at a non-ramification point. In the rest of the boundary divisors, the base is reducible. A generic point of $\Delta_3(g_1, g_2)$ represents two irreducible trigonal curves of genus $g_1$ and $g_2$ glued at a triple ramification point; this union maps to a union of two $\P^1$s glued at a point. Likewise, a generic point of $\Delta_5(g_1, g_2)$ represents an irreducible trigonal curve of genus $g_1$ attached to a $\P^1$ at a non-ramification point and to a hyperelliptic curve of genus $g_2$ at a simple ramification point; this apparatus maps to a union of two $\P^1$s glued at a point as shown in the diagram. We leave the interpretations of the other pictures to the reader. The order of $g_1$ and $g_2$ is important only in $\Delta_4(g_1,g_2)$ and $\Delta_5(g_1,g_2)$.

We often call the divisors $\Delta_i(g_1,g_2)$ (for $1 \leq i \leq 6$) and $H$ \emph{higher boundary divisors}. 
\begin{proposition}
  The classes $\lambda$, $\delta$, and the higher boundary divisors $H$, $\Delta_i(g_1, g_2)$ for $1 \leq i \leq 6$ as shown in \autoref{fig:boundary} form a basis of $\Pic_\Q(\o{\orb H}_g^3)$.
\end{proposition}
\begin{proof}
  We know that $\Pic_\Q({\orb H}^3_g)$ is trivial (see \citep{Bolognesi09:_Stack_Of_Trigon_Curves} or \citep{Stankova-Frenkel00:_Modul_Of_Trigon_Curves}). It follows that $\Pic_\Q(\o{\orb H}_g^3)$ is generated by $A$, $\delta$ and the higher boundary divisors. By test-curve calculations (for example, those in \autoref{sec:divisor}), one can verify that these are independent and that one can use $\lambda$ instead of $A$.
\end{proof}
\section{Evaluating the divisor classes on a family of triple covers}\label{sec:formulas}
In this section, we recall the classical description of triple covers in terms of linear algebraic data on the base and express the divisor classes $\lambda$ and $\delta$ in terms of invariants of this data.

Let $Y$ be an integral scheme and $\phi \from X \to Y$ a finite flat map of degree 3. We have the exact sequence
\[ 0 \to O_Y \to \phi_* O_X \to F \to 0,\]
where $F$ is a locally free sheaf of rank 2 on $Y$. The sequence admits a splitting by $\frac13$ times the trace map $\frac13\tr \from \phi_* O_X \to O_Y$. Set $E = \dual{F}$ and let $p \from \P E \to Y$ be the projectivization. If $\phi$ has Gorenstein fibers then $X$ can be embedded as a Cartier divisor in $\P E$.  In that case, the resolution of $O_X$ as an $O_{\P E}$ module is given by
\[ 0 \to p^* \det E \otimes O_{\P E}(-3) \to O_{\P E} \to O_X .\]
Thus, a triple cover of $Y$ is determined by a locally free sheaf $E$ of rank 2 and a section on $\P E$ of the line bundle $O_{\P E}(3) \otimes p^* \det \dual{E}$, or equivalently by a section on $Y$ of the vector bundle $\Sym^3(E) \otimes \det \dual{E}$. Being local on the base, this description holds also if $Y$ is an algebraic stack.

For the rest of the section, we work with the following setup:
\begin{equation}\label{eq:calculation_setup}
  \begin{split}
  B &= \text{a smooth projective curve,}\\
  \pi \from \orb S \to B &= \text{a generically smooth, proper family     of orbi-nodal curves}\\
\phi \from \orb C \to \orb S &= \text{a Gorenstein triple cover}\\
E &= \dual{(\phi_* O_{\orb C}/O_{\orb S})}.
\end{split}
\end{equation}
Denote by $\omega_{\orb C}$ the relative dualizing sheaf of $\pi \from \orb C \to B$ and set $\pi_{\orb C} = \pi \circ \phi \from \orb C \to B$. Define
\[ \kappa_{\orb C} = {\pi_{\orb C}}_*(\omega_{\orb C}^2) \text{ and } \lambda_{\orb C} = c_1(R{\pi_{\orb C}}_* O_{\orb C}).\]
Define $\omega_{\orb S}$, $\kappa_{\orb S}$ and $\lambda_{\orb S}$ likewise. Our first task is to express $\lambda_{\orb C}$ and $\kappa_{\orb C}$ in terms of the invariants of $\orb S$ and $E$. Note that $\lambda$ and $\kappa$ are unchanged even if we replace the orbifolds by their coarse spaces.

For $a \in \Z/n\Z$, denote by $k(a)$ the vector bundle on $B\mu_n$ corresponding to the character $\zeta \mapsto \zeta^a$. Denote by $\delta_{\orb S}$ the degree of the singular locus of $\pi \from \orb S \to B$ interpreted in the orbifold sense. For example, the orbi-node $[k[u,v,t]/(uv-t)/\mu_n]$ contributes $\frac{1}{n}$. We caution the reader that the relation $12\lambda = \kappa + \delta$ does not hold for orbi-nodal curves.

\begin{proposition}\label{thm:kappa-lambda}
  In the setup of \eqref{eq:calculation_setup}, we have
\begin{align*}
    \kappa_{\orb C} &= 3\kappa_{\orb S}+ 2c_1(E)^2  + 4c_1(E)\omega_{\orb S} - 3c_2(E), \text{ and }\\
\lambda_{\orb C} &= \lambda_{\orb S} + \frac{\kappa_{\orb S}+\delta_{\orb S}}{6} + \frac{c_1(E)^2}{2} + \frac{c_1(E)\omega_{\orb S}}{2} - c_2(E) + \sum_{\text{Orbi }p \in \orb S} \chi(p),
\end{align*}
where $\chi(p)$ is defined as follows. Consider a point $p \in \orb S$ with stabilizer group $\mu_n$. Let the restriction of $E$ to the $B\mu_n$ based at $p$ be $k(-a) \oplus k(-b)$. Let $\zeta$ be a primitive $n$th root of unity. Then
\[ \chi(p) = \frac{1}{n}\sum_{i=1}^{n-1}\left(\frac{\zeta^{ia}+\zeta^{ib}}{2-\zeta^i-\zeta^{-i}}\right).\]
\end{proposition}
\begin{proof}
  In the following computation, pullbacks and push-forwards have been   omitted unless it causes ambiguity. Unless otherwise indicated, dualizing sheaves are relative to $B$. We think of $\orb C$ as embedded in $\P E$ over $\orb S$. Let $\xi$ be the class of $O_{\P   E}(1)$ on $\P E$. Since the embedding $\orb C \into \P E$ is by the relative canonical $\omega_{\orb C/\orb S}$, we have
\[ \omega_{\orb C} = \xi + \omega_{\orb S}.\]
Now, we compute
\begin{align*}
  \kappa_{\orb C} &= \omega_{\orb C}^2 \\
  &= (\xi + \omega_{\orb S})^2|_{\orb C} \\
  &= (\xi^2 + 2\xi\omega_S)|_{\orb C} + \omega_S^2|_{\orb C} \\
  &= (\xi^2 + 2\xi\omega_S)(3\xi-c_1(E)) + 3 \kappa_S\\
  &= 2c_1(E)^2 - 3c_2(E) + 4c_1(E)\omega_S + 3 \kappa_S.
\end{align*}
The last step is mere algebraic simplification using $\xi^2 - c_1(E) \xi + c_2(E) = 0.$

For $\lambda_{\orb C}$, we see that in the Grothendieck group of $B$, we have
\[ R{\pi_{\orb C}}_*O_{\orb C} = R{\pi}_* O_{\orb S} + R{\pi}_* \dual E,\]
and by taking degrees, 
\[ \lambda_{\orb C} = \lambda_{\orb S} + c_1(R{\pi}_*\dual E).\]
To compute the last term, we use Grothendieck--Riemann--Roch. We must be careful, however, as $\pi \from \orb S \to B$ is a Deligne--Mumford stack. From \citep[Corollary~5.2]{edidin12:_rieman_roch_delig_mumfor} and using the notation therein, we have
\begin{equation}\label{eq:grr}
  \ch(R{\pi}_*\dual E) = \pi_*\left(\ch\left(t\left(\frac{f^*\dual E}{\lambda_{-1}(N^*_f)}\right)\right) \td(I\orb S)\right).
\end{equation}
Here $f \from I \orb S \to \orb S$ is the inertia stack, $N^*_f$ the conormal bundle of $f$, the operator $\lambda_{-1}$ is $1 - \Lambda^1 +  \Lambda^2 + \dots$, and the operator $t$ is the `twisting operator.' Let us describe these quantities in our context. The inertia stack $I\orb S$ is a disjoint union
\[ I \orb S = \orb S\ \sqcup\ \bigsqcup_{\text{Orbi }p \in \orb S} (\mu_{n_p} \setminus \{\id\}) \times B\mu_{n_p},\]
where $\mu_{n_p}$ is the automorphism group at $p$. The conormal bundle $N_f$ is trivial on the $\orb S$ component of $I\orb S$ and is the cotangent space of $p$ (with the $\mu_{n_p}$ action) on $\{\zeta\} \times B\mu_{n_p}$. Therefore, $\lambda_{-1}N^*_f$ equals $1$ on $\orb S$ and $(2-k(1)-k(-1))$ on $\{\zeta\} \times B{\mu_{n_p}}$. Finally, the twisting operator on the Grothendieck group acts trivially along the $\orb S$ component and sends $k(a)$ to $\zeta^a k(a)$ along $\{\zeta\} \times B{\mu_{n_p}}$. Say $\dual E$ splits as $k(a_p) \oplus k(b_p)$ on the $B\mu_{n_p}$ at $p$. Then, separating the contributions of the different components of $I\orb S$ in \eqref{eq:grr} gives
\begin{align*}
  \ch(R\pi_*\dual E) &= \pi_*\left(\ch(\dual E) \td(\orb S)\right) + \sum_{\substack{\text{Orbi }p \in \orb S\\ 1 \neq \zeta \in \mu_{n_p}}}  \frac{1}{n_p} \cdot \frac{\zeta^{a_p} +   \zeta^{b_p}}{2-\zeta-\zeta^{-1}}\\
  &= \pi_*\left(2-c_1(E)+\frac{c_1(E)^2-2c_2(E)}{2}\right)\left(1-\frac{c_1(\Omega)}{2} + \frac{c_1(\Omega)^2 + c_2(\Omega)}{12}\right) + \sum_{\text{Orbi } p \in \orb S} \chi(p),
\end{align*}
where $\Omega = \Omega_{\orb S/B}$. Noting that $c_1(\Omega) = \omega_{\orb S}$ and $c_2(\Omega) = \delta_{\orb S}$ yields the formula for $\lambda_{\orb C}$.
\end{proof}

\section{Sweeping curves}\label{sec:curve}
In this section, we prove the first half of \autoref{thm:main} by constructing curves with slope $s_g$ that sweep a Zariski open locus in $\o{\orb   H}^3_g$.
\subsection{The even genus case}\label{sec:even_curve}
Let $g = 2n-2$. Consider a general pencil of curves in the linear system $|(3,n)|$ on $\P^1 \times \P^1$. This gives a curve $\P^1 \to \o{\orb H}^3_g$ and such curves sweep a Zariski open locus of $\o{\orb H}^3_g$. Expressed differently, let $S = \P^1 \times \P^1$ and denote by $\pi$ the second projection. Call $\sigma$ and $F$ the class of a constant section and a fiber of $\pi$ respectively. Let $E$ be the vector bundle
\[ E = O(n \sigma + F) \oplus O(n \sigma + F).\]
Construct a triple cover $\phi \from C \to S$ by letting $C \subset \P E$ be general in the linear system $|3\xi - \det E|$ where $\xi$ is the class of $O_{\P E}(1)$. 
\begin{proposition}\label{thm:even_sweeping}
  The above construction yields curves that sweep a Zariski open subset of $\o{\orb H}^3_g$ and have slope $7 + 6/g$.
\end{proposition}
\begin{proof}
  Since a generic trigonal curve of genus $g = 2n-2$ embeds in $\P^1\times\P^1$ as a divisor of class $(3, n)$, we get the first assertion. For the slope, using the formulas of \autoref{thm:kappa-lambda}, we get
\[ \lambda = g \text{ and } \kappa = 5g-6.\]
Using $12 \lambda = \kappa + \delta$, we get $\delta = 7g+6$.
\end{proof}

\subsection{The odd genus case}\label{sec:odd_curve}
Let $g = 2n-1$. As in the previous case, let $S = \P^1\times\P^1$, and denote by $\pi$ the second projection. Call $\sigma$ and $F$ the class of a constant section and a fiber respectively. Let $E$ be the vector bundle
\[ E = O(n \sigma + F) \oplus O((n+1)\sigma + 2F).\]
Construct a triple cover $\phi \from C \to S$ by letting $C \subset \P E$ be general in the linear system $|3 \xi - \det E|$ where $\xi$ is the class of $O_{\P E}(1)$. 
\begin{proposition}\label{thm:odd_sweeping}
  The above construction yields curves that sweep a Zariski open subset of $\o{\orb H}^3_g$ and have slope $7 + 20/(3g+1)$.
\end{proposition}
\begin{proof}
  Let $b \in B$ be a point. The main observation is that the map of linear systems
  \begin{equation}\label{eqn:surjects}
    H^0(O_{\P E}(3\xi - \det E)) \to H^0(O_{\P E_b}(3\xi - \det     E_b))
\end{equation}
is surjective. Indeed, it suffices to show that
\[ H^0(\Sym^3 E \otimes \det \dual E) \to H^0(\Sym^3 E_b \otimes \det \dual E_b)\]
is surjective. The cokernel of this map is $H^1(\Sym^3E \otimes \det\dual E \otimes O(-F))$. Since 
\[\Sym^3 E \otimes \det \dual E = O((n-1)\sigma) \oplus O(n\sigma+F) \oplus O((n+1)\sigma + 2F) \oplus O((n+2)\sigma+3F),\]
it follows that $H^1(\Sym^3(E) \otimes \det \dual E \otimes O(-F)) = 0$.

Note that away from a locus of codimension two, a divisor in $|3\xi - \det E_b|$ gives a point $\phi \from C \to \P^1$ of $\o{\orb H}^3_g$. Coupled with the surjection \eqref{eqn:surjects}, this implies that a general divisor in $|3\xi - \det E|$ gives a curve in $\o{\orb H}^3_g$. The surjection \eqref{eqn:surjects} also implies that these curves cover a Zariski open subset of $\o{\orb H}^3_g$. 

It remains to check the slope. By \autoref{thm:kappa-lambda}, we get
\[ \lambda = 3n - 1 \text{ and } \kappa = 15n-15.\]
Using $12 \lambda = \kappa + \delta$, we get $\delta = 21n+3$. Hence the slope is 
\[\frac{\delta}{\lambda} =  \frac{21n+3}{3n-1} = 7 + \frac{20}{3g+1}.\]
\end{proof}

\section{Effective divisors}\label{sec:divisor}
In this section, we prove the second half of \autoref{thm:main}. The idea is to exhibit an effective divisor with class proportional to $(s_g \lambda - \delta)$.

\subsection{The even genus case}\label{sec:divisor_even}
Let $\mu \subset \o{\orb H}^3_g$ be the closure of the locus in $\orb H^3_g$ consisting of $[\phi \from C \to \P^1]$ for which the rank two bundle $\phi_* O_C/O_{\P^1}$ is unbalanced. It is easy to see that $\mu$ is a divisor. It is called the \emph{Maroni divisor}. The divisor $\mu$ plus an effective sum of boundary divisors is a positive multiple of $(7+6/g) \lambda - \delta$. To prove this, we compute the class of $\mu$
\[ \mu = a \lambda - b \delta - c H - \sum c_i(g_1,g_2) \Delta_i(g_1,g_2),\]
by a series of test curve calculations.

\subsubsection*{Setup of test curve calculations} All the test curves are of the following format:
\begin{enumerate}
\item Start with a fibration of rational (orbi)-curves $\pi \from \orb S \to B$ 
\item Pick a sufficiently $\pi$-positive rank two bundle $E$ on $\orb S$. 
\item Construct $\phi \from \orb C \to \orb S$ by picking a generic
  divisor in $|O_{\P E}(3)\otimes \dual{\det E}|$.
\end{enumerate}
Thus, the test-curve is specified by $\orb S \to B$ and $E$. In all the test-curves, we have $B = \P^1$. The family $\orb S \to B$ is one of the following four:
\begin{enumerate}
\item $\orb S_0 = \P^1 \times B$. It is used to compute the
  coefficients of $\lambda$ and $\delta$.
\item $\orb S_1$ is obtained by blowing up $\P^1 \times B$ at a point. It is used to compute the coefficients of $\Delta_1(g_1,g_2)$, $\Delta_4(g_1,g_2)$, $\Delta_6(g_1,g_2)$ and $H$.
\item $\orb S_2$ is obtained by blowing up $\P^1 \times B$ along a $k[\epsilon]/\epsilon^2$ transverse to a fiber and replacing the resulting $A_2$ singularity $k[x,y]/(xy-t^2)$ by the orbi-node $[k[u,v]/(uv-t) /
  \mu_2]$. It is used to compute the coefficients of $\Delta_2(g_1,g_2)$ and $\Delta_5(g_1,g_2)$.
\item $\orb S_3$ is obtained by blowing up $\P^1 \times B$ along a   $k[\epsilon]/\epsilon^3$ transverse to a fiber and replacing the   resulting $A_2$ singularity $k[x,y]/(xy-t^3)$ by the orbi-node   $[k[u,v]/(uv-t) / \mu_3]$. It is used to compute the coefficients of   $\Delta_3(g_1,g_2)$.
\end{enumerate}
For all $\orb S_i$, denote by $s$ the preimage of a constant section from $\P^1 \times B$ not passing through the center of the blowup. For $\orb S_i$ with $i \geq 1$, let $F_\infty$ (resp. $F_0$) be the component of the reducible fiber that does not (resp. does) intersect $s$.  Note that 
\[ s^2 = 0, \quad F_0 \cdot F_\infty = \frac{1}{i} \text{ and } F_0^2 = F_\infty^2 = \frac{-1}{i}.\]
Also, 
\[\omega_{\orb S_0/B} = -2s \text{ and }\omega_{\orb S_i/B} =
-2s + i F_\infty \text{ for $i \geq 1$.}\]

\autoref{table:test-curves-mu} lists various test-curves, their intersection with the higher boundary divisors and the value of the \emph{Residual}, defined by
\[ \text{Residual} = (7g+6)\lambda - g \delta - 2(g-3)\mu.\] The divisors $\lambda$ and $\delta$ are computed by applying \autoref{thm:kappa-lambda} and using $12 \lambda = \kappa + \delta$ for the induced family $\pi \from C \to B$, where $C$ is the coarse space of $\orb C$. The orbifold correction term in the expression for $\lambda$ is $0$ for $\orb S_1$, $0$ for $\orb S_2$ and $-2/9$ for $\orb S_3$. These curves intersect the higher boundary divisors in general points (\autoref{thm:generic_boundary}) and hence the central fiber does not contribute to $\mu$. Note that $\mu$ is $0$ on all the test-curves except the second, where it is $(\deg L - \deg M)$. 

We take $g = 2n-2$, and $L$, $M$ to be sufficiently positive line bundles pulled from $B$.
\renewcommand{\arraystretch}{2}
\begin{longtable}{| l | p{.35\textwidth} | p{.17\textwidth} |                     p{.35\textwidth} |}
\caption{Test-curve calculations for the class of $\mu$ (even $g$)}\\
\hline
$\orb S$ & $E$ & Intersection with higher boundary& Residual\\
\hline
\endhead
\hline
\endfoot
$\orb S_0$ & $L(n)^{\oplus 2}$& --- & 0\\

$\orb S_0$ & 
Generic extension \newline 
$L((n-1)s) \into E \onto M((n+1)s).$ & 
---&
$0$\\

$\orb S_1$ &
$L(ns - bF_\infty)\oplus M(ns-bF_\infty)$ \newline
with $1 \leq b \leq n-1$ &
$\Delta_1(g_1,g_2) = 1$\newline
$g_1 = 2(n-b)-2$\newline
$g_2 = 2b - 2$ &
$\frac{3}{2}g_1g_2$
\\

$\orb S_1$ &
${L(ns - bF_\infty)}\oplus {M(ns - (b-1)F_\infty)}$\newline
with $2 \leq b \leq n-1$&
$\Delta_1(g_1,g_2) = 1$\newline
$g_1 = 2(n-b)-1$\newline
$g_2 = 2b - 3$ &
$\frac{1}{2} (3g_1g_2+g_1+g_2-1)$
\\

$\orb S_1$ &
Generic extension ($ 2 \leq b \leq 2n-2$)\newline
$L(ns-bF_\infty) \into E \onto M(ns)$&
$\Delta_4(g_1,g_2) = 1$\newline
$g_1 = 2n-b-2$\newline
$g_2 = b-1$&
$\frac{1}{2}g_2(g_1g_2+g_2^2+5g_1-1)$
\\

$\orb S_1$ &
Extension (with $n \leq b \leq 2n-2$)\newline
$L(ns-bF_\infty) \into E \onto M(ns)$ \newline
such that $E = O \oplus O(2n-b)$ on $F_0$&
$\Delta_6(g_1,g_2) = 1$\newline
$g_1 = 2n-b-1$\newline
$g_2 = b-1$&
$\frac{1}{2}g_2(g_1g_2+g_2^2+5g_1-g_2-6) + g$
\\

$\orb S_1$ &
$b = 2n -1$ in the above case.
&
$H = 2$&
$\frac{1}{2}g(g-2)(g+1)$
\\

$\orb S_2$ &
$L(ns - bF_\infty) \oplus M(ns-(b-1)F_\infty)$ &
$\Delta_2(g_1,g_2) = 1$\newline
$g_1 = 2n-b-1$\newline
$g_2 = b - 2$&
$3g_1g_2$
\\

$\orb S_2$ &
Generic extension (with $b$ odd)\newline
$L(ns - bF_\infty) \into E \onto M(ns)$ &
$\Delta_5(g_1,g_2) = 1$\newline
$g_1 = 2n-\frac{b+3}{2}$\newline
$g_2 = \frac{b - 1}{2}$ &
$g_2^3+g_1g_2^2-2g_2^2+4g_1g_2-g_1-g_2$
\\

$\orb S_3$ &
$L(ns - bF_\infty) \oplus M(ns-(b-1)F_\infty)$\newline
with $b \equiv 2 \pmod 3$&
$\Delta_3(g_1,g_2) = 1$\newline
$g_1 = 2n-\frac{2b+2}{3}$\newline
$g_2 = \frac{2b - 4}{3}$&
$\frac{9}{2}g_1g_2-g_1-g_2$
\\

$\orb S_3$ &
$L(ns - bF_\infty) \oplus M(ns-(b-2)F_\infty)$\newline
with $b \equiv 1 \pmod 3$&
$\Delta_3(g_1,g_2) = 1$\newline
$g_1 = 2n-\frac{2b+1}{3}$\newline
$g_2 = \frac{2b - 5}{3}$&
$\frac{1}{2}(9g_1g_2-g_1-g_2-3)$
\label{table:test-curves-mu}
\end{longtable}
\begin{remark}\label{rem:even_adjustments}
  \autoref{thm:kappa-lambda} along with $\delta = 12\lambda - \kappa$ gives $\delta$ for the family $\pi \from C \to B$, where $C$ is the coarse space of $\orb C$. In the following cases, however, the special fiber of $C \to B$ has unstable rational tails, which must be contracted to get a family of Deligne--Mumford stable curves. We must accordingly adjust the value of $\delta$.
\begin{enumerate}
\item When the special fiber is in $\Delta_4(g_1, g_2)$, its rational tail must be contracted, which is a $-1$ curve on $C$. Hence we must subtract $1$ from the $\delta$ computed for $\pi \from C \to B$.
\item When the special fiber is in $\Delta_5(g_1, g_2)$, its rational tail must be contracted, which removes an $A_1$ singularity from $C$. Hence we must subtract $2$ from the $\delta$ computed for $\pi \from C \to B$.
\item When the special fiber is in $\Delta_6(g_1, g_2)$, its two rational tails must be contracted, which are both $-1$ curves on $C$. Hence we must subtract $2$ from the $\delta$ computed for $\pi \from C \to B$.
\item To get a central fiber in $H$ (by setting $b = 2n-1$ in the fifth row), we must contract $F_\infty$ and the two rational curves over it. Furthermore, to get a Deligne--Mumford stable model, we must also contract the rational tail on the remaining curve. Thus, in total, we must subtract 3 from the $\delta$ computed for $\pi \from C \to B$.
\end{enumerate}
\end{remark}
It remains to argue that the curves constructed in \autoref{table:test-curves-mu} indeed intersect the boundary divisors in generic points. This is immediate for $\Delta$. The following proposition proves it for the higher boundary.
\begin{proposition}\label{thm:generic_boundary}
  The special fiber of the family of curves in \autoref{table:test-curves-mu} is generic in the boundary divisor indicated in the third column.
\end{proposition}
\begin{proof}
  Let $\pi \from \orb S_i \to B$ (for $i \geq 1$) and $E$ be as
  indicated in \autoref{table:test-curves-mu}. Let $0 \in B$
  correspond to the special fiber. Let $V \subset H^0(\Sym^3 E \otimes
  \dual {\det E}|_0)$ be the image of $H^0(\Sym^3 E \otimes \dual{\det
  E})$. We must prove that the trigonal curves given by (generic)
  elements of $V$ make up a dense subset in the corresponding boundary
  divisor. By choosing $L$ and $M$ sufficiently ample, we are reduced
  to proving this statement where $V$ is the image of $\pi^*\pi_*
  (\Sym^3 E \otimes \dual{\det E})$ in $H^0(\Sym^3 E \otimes \dual{\det E}|_0)$.
  The question is thus local around $0$.
  
  Consider a general point $[\phi \from C \to P]$ of
  $\Delta_i(g_1,g_2)$ with $P = F_0 \cup F_\infty$ (the case of $H$ is handled in this guise by letting $i = 6$ and $g_1 = 0$.) Note that line bundles on $P$ have integral degree; rank two bundles split; and for generic $[\phi \from C \to P]$ the splitting of $E$ is indeed as obtained in \autoref{table:test-curves-mu}. Let $(B,0)$ be a germ of a smooth curve and $\pi \from \orb S \to B$ with $\phi \from \orb C \to \orb S$ a smoothing of $\phi \from C \to P$ with smooth $\orb S$, and where the generic fiber is a smooth, balanced trigonal curve. We show that such a smoothing can be obtained in the curves in \autoref{table:test-curves-mu}.

  Let $s \from B \to \orb S$ be a section intersecting $F_0$. Set $E =
  \dual{(\phi_* O_{\orb C}/O_{\orb S})}$ and $L = O(nS -
  bF_\infty)$. To get such a smoothing in
  \autoref{table:test-curves-mu}, $E$ must be expressible as an
  extension
  \[ 0 \to L \to E \to M \to 0.\]
  It suffices to show that there is a nowhere vanishing map $L \to E$. Set $G = \dual L \otimes E$. For the case of $\Delta_1$, $\Delta_2$, and $\Delta_3$, where both branches of $\phi \from C \to P$ are connected, and hence balanced by genericity, we have $H^1(G|_0) = 0$. Therefore, in this case, $\pi^*\pi_*G \to H^0(G|_0)$ is surjective. Since $G|_0$ has a nowhere vanishing section, we conclude that $G$ has a nowhere vanishing section.

The case of $\Delta_4$, $\Delta_5$ and $\Delta_6$ is a bit subtle. Think of a section of $G$ as a divisor of class $\sigma = O_{\P G}(1)$ in $\P G$; a nowhere vanishing section corresponds to a divisor that does not contain a fiber of $\P G \to \orb S$. On a generic fiber $\P^1_t$ of $\orb S \to B$, the class $\sigma$ is that of a constant section of $\P G_t \cong \P^1 \times \P_t^1$. Consider a family of divisors $X_t \subset \P G_t$ in class $\sigma$ and let $X_0 \subset \P G|_0$ be the flat limit. Then $X_0$ is a Cartier divisor in $\P G|_0$ of class $\sigma - aN$, where $N = \P G|_{F_\infty}$ and $a \in \Z$. It suffices to show that $a = 0$ and $X_0$ does not contain any fibers of $\P G|_0 \to P$. 

 Let the preimage of $C \to P$ over $F_\infty$ be $C_1 \sqcup C_2$ where $C_1 \isom F_\infty$ is degree one and $C_2 \to F_\infty$ is degree two. Note that $N \cdot C_2 \neq 0$ but $\sigma \cdot C_2 = 0$. Hence $a = 0$ if $X_0$ does not intersect $C_2$. Therefore it suffices to show that $X_0$ does not intersect $C_2$ and does not contain any fibers of $\P G|_{F_0} \to F_0$. 

Consider the highly singular trigonal curve $C' \to \P^1$ obtained by contracting the components $C_1$ and $C_2$ of $C$ as follows. In the case where $C$ is an orbi-curve, we first pass to the coarse space. The rational component $C_1$ is contracted to the (smooth) point of attachment. The hyperelliptic component $C_2$ of genus $g_2$ is contracted to create a $A_{2g_2+1}$ singularity for $\Delta_4$ and $\Delta_6$ and a $A_{2g_2}$ singularity for $\Delta_5$ (see \autoref{fig:crimping}). Note that there is a unique way to carry out such crimping while keeping a map to $\P^1$. It is not hard to check that for a generic central fiber $[\phi \from C \to P]$, the resulting $C' \to \P^1$ is a balanced triple cover.
\begin{figure}[ht]
  \subfigure[$\Delta_4$ or $\Delta_6$]{
      \begin{tikzpicture}[yscale=.15, xscale=.6, thick]
        \path[dotted] (-1.5,-8) rectangle (3.5,4);
        \begin{pgflowlevelscope}{\pgftransformrotate{20}}
          \draw [smooth, tension=1.5, dotted]
          plot coordinates{(-1,1) (0,1)}
          plot coordinates{(-1,0) (0,0)}
          plot coordinates{(-1,-1) (0,-1)};
          \draw [smooth, tension=1.5]
          plot coordinates{(0,1) (1,1)}
          plot coordinates{(0,0) (1,0)}
          plot coordinates{(0,-1) (1,-1)};
          \draw (-1,-5) -- (.7, -5);
          \path (-.5,-2.7) edge [->, thin] (-.5,-4.2);
        \end{pgflowlevelscope}
        \begin{pgflowlevelscope}{ 
            \pgftransformxscale{-1}
            \pgftransformxshift{-1cm}           
            \pgftransformrotate{20}}
          \draw [smooth, tension=2]
          plot coordinates{(1.1,1) (0, .2) (1.1,0)}
          plot coordinates{(-1,.8) (0, .2) (-1,0)}
          plot coordinates{(1,-1) (-1, -1)};
          \draw (-1,-5) -- (.7, -5);
          \path (-.5,-2.7) edge [->, thin] (-.5,-4.2);
        \end{pgflowlevelscope}
        \draw (3.2,0) node  {$g_2$};
        \begin{scope}[xshift=6cm, yshift=-1.5cm]
          \draw (-2.5,0) edge [thin,->,decorate,decoration={snake, amplitude=.05cm}] (-1.5,0);
          \draw [smooth, tension=1.5, dotted]
          plot coordinates{(-1,1) (0,1)}
          plot coordinates{(-1,0) (0,0)}
          plot coordinates{(-1,-1) (0,-1)};        
          \draw [smooth, tension=0.5]
          plot coordinates{(0,1) (.5,.5) (1,1)}
          plot coordinates{(0,0) (.5,.5) (1,0)}
          plot coordinates{(0,-1) (1,-1)};
          \path (0,-2.7) edge [->, thin] (0,-4.2);
          \draw (1,3) node {\tiny $y^2=x^{2g_2+2}$};
          \draw (-1,-5) -- (1, -5);
        \end{scope}
      \end{tikzpicture}
    }\qquad
  \subfigure[$\Delta_5$]{
      \begin{tikzpicture}[yscale=.15, xscale=.6, thick]
        \path[dotted] (-1.5,-8) rectangle (3.5,4);
        \begin{pgflowlevelscope}{\pgftransformrotate{20}}
          \draw [smooth, tension=1.5, dotted]
          plot coordinates{(-1,1) (0,1)}
          plot coordinates{(-1,0) (0,0)}
          plot coordinates{(-1,-1) (0,-1)};
          \draw [smooth, tension=1.5]
          plot coordinates{(0,1) (.8,.5) (0,0)}
          plot coordinates{(0,-1) (1,-1)};
          \draw (-1,-5) -- (.7, -5);
          \path (-.5,-2.7) edge [->, thin] (-.5,-4.2);
        \end{pgflowlevelscope}
        \begin{pgflowlevelscope}{ 
            \pgftransformxscale{-1}
            \pgftransformxshift{-1cm}           
            \pgftransformrotate{20}}
          \draw [smooth, tension=2]
          plot coordinates{(-1,.8) (1, .2) (-1,0)}
          plot coordinates{(1,-1) (-1, -1)};
          \draw (-1,-5) -- (.7, -5);
          \path (-.5,-2.7) edge [->, thin] (-.5,-4.2);
        \end{pgflowlevelscope}
        \draw (3.2,0) node  {$g_2$};
        \begin{scope}[xshift=6cm, yshift=-1.5cm]
          \draw (-2.5,0) edge [thin,->,decorate,decoration={snake, amplitude=.05cm}] (-1.5,0);
          \draw [smooth, dotted]
          plot coordinates{(-1,1) (0,1)}
          plot coordinates{(-1,0) (0,0)}
          plot coordinates{(-1,-1) (0,-1)};        
          \draw [smooth, tension=2]
          plot coordinates{(0,1) (.7,.5) (.7,.5) (0,0)}
          plot coordinates{(0,-1) (1,-1)};
          \path (0,-2.7) edge [->, thin] (0,-4.2);
          \draw (1,3) node {\tiny $y^2=x^{2g_2+1}$};
          \draw (-1,-5) -- (1, -5);
        \end{scope}
      \end{tikzpicture}
    }
  \caption{Crimping $C \to P$ to get a singular $C' \to \P^1$}
  \label{fig:crimping}
\end{figure}

  Let $\orb S \to \P^1 \times B$ be the contraction of $F_\infty$ and $\orb C \to \orb C'$  the contraction of the curves in $\orb C$ over $F_\infty$. Then $\orb C' \to \P^1 \times B$ is a triple cover; in fact, it is the unique extension to $\P^1 \times B$ of the cover $\orb C \setminus \phi^{-1}(F_\infty) \to \orb S \setminus F_\infty$. The central fiber of $\orb C' \to \P^1 \times B$ is $C' \to \P^1$, by construction. Define the analogues $E'$ and $G'$ of $E$ and $G$. Consider the flat limit $X'_0$ of $X_t$ in $\P G'|_0$. To show that $X_0$ does not intersect $C_2$ and does not contain fibers of $\P G|_{F_0}\to F_0$, it suffices to show that $X'_0$ stays away from the singularity of $C'$ and does not contain fibers of $\P G'|_0 \to \P^1$. Since $G'$ is balanced, the limit $X'_0$ is simply a constant section of $\P G'|_0 \cong \P^1 \times \P^1$, which will stay away from the singularity of $C'$ for a generic choice of the family $X_t$. The proof is thus complete.
\end{proof}

\begin{proposition}\label{thm:class_mu}
  Let $g \geq 4$ be even. In the rational Picard group of $\o{\orb H}^3_g$, we have
  \[ 2(g-3)[\mu] = (7g+6)\lambda - g \delta - c H - \sum c_i(g_1,g_2)\Delta_i(g_1,g_2),\]
where $c \geq 0$ and $c_i \geq 0$. In particular, a positive multiple of $(7g+6)\lambda - g \delta$ is equivalent to an effective divisor.
\end{proposition}
\begin{proof}
  Intersect both sides of the equation by each of the test-curves in \autoref{table:test-curves-mu}. The coefficients $c$ and $c_i(g_1,g_2)$ are precisely the entries in the ``Residual'' column. It is easy to verify that these entries are non-negative for the allowed values of $g_1$, $g_2$, given $g \geq 4$.
\end{proof}
\begin{corollary}\label{thm:sweeping_slope_sharp_even}
  Let $g \geq 4$ be even. Let $B \to \o T_g$ be a curve such that a generic point of $B$ corresponds to a smooth trigonal curve not contained in the image of $\mu$. Then 
\[ (7g+6) (\lambda \cdot B) \geq g (\delta \cdot B).\]
\end{corollary}
\begin{proof}
  Lift $B \to \o T_g$ to $B \to \o{H}^3_g$. Then a generic point of $B$ lies in $H^3_g \setminus \mu$.  We conclude that $B$ intersects non-negatively with $\mu$ and with all the boundary divisors. From \autoref{thm:class_mu}, we get
\[ (7g+6) (\lambda \cdot B) \geq g (\delta \cdot B).\]
\end{proof}

\subsection{The odd genus case}\label{sec:divisor_odd}
Let $\tau \subset \o {\orb H}^3_g$ be the closure of the locus consisting of $[\phi \from C \to \P^1]$ for which the scroll $\P E$ is isomorphic to $\F_1$ and the image of the embedding $C \into \P E$ is tangent to the directrix, where as usual $E = \dual{(\phi_*O_C/O_{\P^1})}$. Then $\tau$ is a divisor. We call it the \emph{tangency divisor}. In the odd genus case, it plays the role of the Maroni divisor. The main task is computing its class in the Picard group of $\o {\orb H}^3_g$.

The setup of the test-curve calculations is the same as that in the even genus case (\autoref{sec:divisor_even}). In particular, we retain the meanings of $B$, $\orb S_i$ (for $0 \leq i \leq 3$), $s$, $F_0$ and $F_\infty$. We take $g = 2n-1$ and $L$, $M$ sufficiently ample line bundles pulled from $B$.

\autoref{table:test-curves-tau} lists various test-curves, their intersection with the relevant boundary divisors and the value of the \emph{Residual}, defined by
\[ \text{Residual} = (21g+27)\lambda - (3g+1)\delta - 2\tau.\] The divisors $\lambda$ and $\delta$ are computed by applying \autoref{thm:kappa-lambda} and using $12 \lambda = \kappa + \delta$ for the induced family $\pi \from C \to B$, where $C$ is the coarse space of $\orb C$. The divisor $\tau$ is calculated as follows. We pick a section $\sigma \from \orb S \to \P E$, given by a quotient $E \to Q$, where $Q$ is a line bundle on $\orb S$, such that $\sigma$ agrees with the directrix on the generic fiber. We consider the curve $D = \sigma(\orb S) \cap \orb C$. The points of $\tau$ on the base $B$ are simply the branch points of $D \to B$. To compute their number, we note that the class of $D$ on $\orb S$ is given by
\begin{align*}
  [D] &= \sigma^* [\orb C]  \\
  &= \sigma^* (3 \xi - c_1(E)) = 3 c_1(Q) - c_1(E).
\end{align*}
By adjunction, we get
\begin{equation}\label{eqn:tau}
  \begin{split}
  \tau &= c_1(\omega_{D/B}) \\
  &= (3c_1(Q) - c_1(E)) \cdot (3c_1(Q) - c_1(E) + c_1(\omega_{\orb     S/B})).
\end{split}
\end{equation}
The curves in \autoref{table:test-curves-tau} intersect the higher boundary divisors in general points. Hence, the special fiber does not contribute to $\tau$. The formula \eqref{eqn:tau}, however, includes a contribution from the special fiber for some test-curves, which we must correct. We comment upon these adjustments in \autoref{rem:odd_adjustments}. In \autoref{table:test-curves-tau}, the role of $Q$ is always played by the bundle of the form $M(\cdots)$.
\renewcommand{\arraystretch}{2}
\begin{longtable}{| l | p{.38\textwidth} | p{.17\textwidth} | p{.35\textwidth} |}
\caption{Test-curve calculations for the class of $\tau$ (odd $g$)}\\
\hline
$\orb S$ & $E$ & Intersection with higher boundary & Residual\\
\hline
\endhead
\hline
\endfoot
$\orb S_0$ & 
$E = L((n+1)s) \oplus M(ns)$ & 
---&
$0$\\

$\orb S_1$ &
$L((n+1)s - bF_\infty)\oplus M(ns-bF_\infty)$ \newline
with $1 \leq b \leq n-1$ &
$\Delta_1(g_1,g_2) = 1$\newline
$g_1 = 2(n-b)-1$\newline
$g_2 = 2b - 2$ &
$\frac{3}{2}g_2(3g_1+1)$
\\

$\orb S_1$ &
Generic extension ($ 2 \leq b \leq 2n-1$)\newline
$L((n+1)s-bF_\infty) \into E \onto M(ns)$&
$\Delta_4(g_1,g_2) = 1$\newline
$g_1 = 2n-b-1$\newline
$g_2 = b-1$&
$\frac{3}{2}g_2(g_1g_2+g_2^2+5g_1+g_2+4)$
\\

$\orb S_1$ &
Extension (with $n+1 \leq b \leq 2n-1$)\newline
$L((n+1)s-bF_\infty) \into E \onto M(ns)$ \newline
such that $E = O \oplus O(2n-b+1)$ on $F_0$&
$\Delta_6(g_1,g_2) = 1$\newline
$g_1 = 2n-b$\newline
$g_2 = b-1$&
$\frac{3}{2}g_2(g_1g_2+g_2^2+5g_1-1) + 3g+1$
\\

$\orb S_1$ &
$b = 2n$ in the above case.
&
$H = 2$&
$\frac{3}{2}g(g^2+3)+2$

\\
$\orb S_2$ &
$L((n+1)s - bF_\infty) \oplus M(ns-(b-1)F_\infty)$ &
$\Delta_2(g_1,g_2) = 1$\newline
$g_1 = 2n-b$\newline
$g_2 = b - 2$ &
$9g_1g_2$
\\

$\orb S_2$ &
Generic extension (with $b$ odd)\newline
$L((n+1)s - bF_\infty) \into E \onto M(ns)$ &
$\Delta_5(g_1,g_2) = 1$\newline
$g_1 = 2n-\frac{b+1}{2}$\newline
$g_2= \frac{b-1}{2} $ &
$3g_2(g_2^2+g_1g_2+4g_1-g_2+4)-3g-1$
\\

$\orb S_3$ &
$L((n+1)s - bF_\infty) \oplus M(ns-(b-1)F_\infty)$\newline
with $b \equiv 2 \pmod 3$&
$\Delta_3(g_1,g_2) = 1$\newline
$g_1 = 2n-\frac{2b-1}{3}$\newline
$g_2 = \frac{2b - 4}{3}$ &
$\frac{3}{2}(9g_1g_2-2g_1-g_2)-1$
\\

$\orb S_3$ &
$L((n+1)s - bF_\infty) \oplus M(ns-(b-2)F_\infty)$\newline
with $b \equiv 1 \pmod 3$&
$\Delta_3(g_1,g_2) = 1$\newline
$g_1 = 2n-\frac{2b-2}{3}$\newline
$g_2 = \frac{2b - 5}{3}$ &
$\frac{3}{2}(9g_1g_2-g_1-2g_2)-1$
\label{table:test-curves-tau}
\end{longtable}
\begin{remark}\label{rem:odd_adjustments}
  In the following cases, we have to make adjustments to $\delta$ and
  $\tau$. The adjustments for $\delta$ are for the same reasons as in
  the case of even genus (\autoref{rem:even_adjustments}). For the
  adjustments to $\tau$, recall the notation in \eqref{eqn:tau}:
  $\sigma \from \orb S \to \P E$ is given by $E \to M(\cdots)$; we set
  $D = \sigma(\orb S) \cap \orb C$ and get $\tau$ by counting the
  number of branch points of $D \to B$ using adjunction.
  \begin{enumerate}
  \item When the special fiber is in $\Delta_4(g_1,g_2)$, the curve
    $D$ includes the rational tail on the central fiber. This tail
    unnecessarily contributes $-2$ in the adjunction formula,
    which we must correct. Also, we must subtract $1$ from the
    $\delta$ counted for $\pi \from C \to B$.
  \item When the special fiber is in $\Delta_5(g_1,g_2)$, the curve
    $D$ includes the orbifold rational tail on the central fiber. This
    tail unnecessarily contributes $-3/2$ in the adjunction formula,
    which we must correct. Also, we must subtract $2$ from the
    $\delta$ counted for $\pi \from C \to B$.
  \item When the special fiber is in $\Delta_6(g_1,g_2)$, the curve
    $D$ includes the rational tail on the central fiber. This
    tail unnecessarily contributes $-2$ in the adjunction formula,
    which we must correct. Also, we must subtract $2$ from the
    $\delta$ counted for $\pi \from C \to B$.
  \item When the central fiber is in $H$ (in the guise of
    $\Delta_6(g_1,0)$ in the fifth row), the curve $D$ includes the
    rational tail on the central fiber. This tail unnecessarily
    contributes $-2$ in the adjunction formula, which we must
    correct. Also, we must subtract $3$ from the $\delta$ counted for
    $\pi \from C \to B$.
  \end{enumerate}
\end{remark}

\begin{proposition}
  The central fiber of the family of curves in
  \autoref{table:test-curves-tau} is generic in the boundary divisor
  indicated in the third column.
\end{proposition}
\begin{proof}
  The proof is the same as that of \autoref{thm:generic_boundary},
  with one modification in the second half (dealing with $\Delta_4$,
  $\Delta_5$ and $\Delta_6$). Consider the highly singular trigonal
  curve $C' \to \P^1$ of genus $g$ obtained by contracting the
  rational and the hyperelliptic tail of the special fiber $\phi \from
  C \to P$. As a part of the genericity of the central fiber $\phi
  \from C \to P$, we must not only assume that $C' \to \P^1$ is
  balanced, but also that the $A_k$ singularity of $C'$ is away from
  the directrix in the embedding $C' \into \F_1$. The rest of the proof
  is almost verbatim.
\end{proof}

We readily deduce the analogues of \autoref{thm:class_mu} and \autoref{thm:sweeping_slope_sharp_even}.
\begin{proposition}\label{thm:class_tau}
  Let $g \geq 5$ be odd. In the rational Picard group of $\o{\orb H}_g^3$, we have
  \[ 2[\tau] = (21g+27)\lambda - (3g+1)\delta - cH -\sum c_i(g_1,g_2)\Delta_i(g_1,g_2),\]
  where $c \geq 0$ and $c_i \geq 0$. In particular, a positive multiple of $(21g+27)\lambda - (3g+1)\delta$ is equivalent to an effective divisor.
\end{proposition}
\begin{corollary}\label{thm:sweeping_slope_sharp_odd}
  Let $g \geq 5$ be odd. Let $B \to \o T_g$ be a curve such that a
  generic point of $B$ corresponds to a smooth trigonal curve not
  contained in the image of $\tau$. Then
  \[ (21g+27) (\lambda \cdot B) \geq (3g+1) (\delta \cdot B).\]
\end{corollary}

\bibliographystyle{abbrvnat}
\bibliography{TrigonalSlopesBib}

\begin{thebibliography}{13}
\providecommand{\natexlab}[1]{#1}
\providecommand{\url}[1]{\texttt{#1}}
\expandafter\ifx\csname urlstyle\endcsname\relax
  \providecommand{\doi}[1]{doi: #1}\else
  \providecommand{\doi}{doi: \begingroup \urlstyle{rm}\Url}\fi

\bibitem[Abramovich et~al.(2003)Abramovich, Corti, and Vistoli]{acv:03}
D.~Abramovich, A.~Corti, and A.~Vistoli.
\newblock Twisted bundles and admissible covers.
\newblock \emph{Comm. Algebra}, 31\penalty0 (8):\penalty0 3547--3618, 2003.

\bibitem[Barja and Stoppino(2009)]{barja09:_slopes}
M.~A. Barja and L.~Stoppino.
\newblock {Slopes of trigonal fibred surfaces and of higher dimensional
  fibrations.}
\newblock \emph{{Ann. Sc. Norm. Super., Cl. Sci. (V)}}, 2009.

\bibitem[{Beorchia} and {Zucconi}(2012)]{beorchia12:_m}
V.~{Beorchia} and F.~{Zucconi}.
\newblock {On the slope conjecture for the fourgonal locus in
  $\overline{M}_g$}.
\newblock \emph{ArXiv e-prints}, Sept. 2012.

\bibitem[Bolognesi and Vistoli(2012)]{Bolognesi09:_Stack_Of_Trigon_Curves}
M.~Bolognesi and A.~Vistoli.
\newblock Stacks of trigonal curves.
\newblock \emph{Trans. Amer. Math. Soc.}, Feb. 2012.

\bibitem[Cornalba and Harris(1988)]{cornalba88:_divis}
M.~Cornalba and J.~Harris.
\newblock Divisor classes associated to families of stable varieties, with
  applications to the moduli space of curves.
\newblock \emph{Ann. Sci. \'Ecole Norm. Sup. (4)}, 21\penalty0 (3):\penalty0
  455--475, 1988.

\bibitem[{Deopurkar}(2012)]{deopurkar12:_compac_hurwit}
A.~{Deopurkar}.
\newblock {Compactifications of Hurwitz spaces}.
\newblock \emph{ArXiv e-prints}, June 2012.

\bibitem[{Edidin}(2012)]{edidin12:_rieman_roch_delig_mumfor}
D.~{Edidin}.
\newblock {Riemann-Roch for Deligne-Mumford stacks}.
\newblock \emph{ArXiv e-prints}, May 2012.

\bibitem[Eisenbud and Harris(1987)]{eisenbud87:_kodair}
D.~Eisenbud and J.~Harris.
\newblock The {K}odaira dimension of the moduli space of curves of genus {$\geq
  23$}.
\newblock \emph{Invent. Math.}, 90\penalty0 (2):\penalty0 359--387, 1987.

\bibitem[Fedorchuk and
  Jensen(2012)]{fedorchuk12:_stabil_hilber_point_canon_curves}
M.~Fedorchuk and D.~Jensen.
\newblock Stability of 2nd hilbert points of canonical curves.
\newblock \emph{Int. Math. Res. Not. IMRN}, 2012.

\bibitem[Fedorchuk and
  Smyth(2012)]{fedorchuk12:_alter_compac_modul_spaces_curves}
M.~Fedorchuk and D.~I. Smyth.
\newblock Alternate compactifications of moduli spaces of curves.
\newblock In Farkas and Morrison, editors, \emph{Handbook of Moduli}.
  International Press, 2012.

\bibitem[Harris and
  Mumford(1982)]{harris82:_kodair_dimen_of_modul_space_of_curves}
J.~Harris and D.~Mumford.
\newblock On the {K}odaira dimension of the moduli space of curves.
\newblock \emph{Invent. Math.}, 67\penalty0 (1):\penalty0 23--88, 1982.

\bibitem[Hassett(2003)]{hassett03:_modul}
B.~Hassett.
\newblock Moduli spaces of weighted pointed stable curves.
\newblock \emph{Adv. Math.}, 173\penalty0 (2):\penalty0 316--352, 2003.

\bibitem[Stankova-Frenkel(2000)]{Stankova-Frenkel00:_Modul_Of_Trigon_Curves}
Z.~E. Stankova-Frenkel.
\newblock Moduli of trigonal curves.
\newblock \emph{J. Algebraic Geom.}, 9\penalty0 (4):\penalty0 607--662, 2000.

\end{thebibliography}

\end{document}